\documentclass[12pt]{scrartcl}

\usepackage{amsmath,amsfonts,amsthm,amssymb,graphicx,float,tikz,indentfirst,graphics,setspace,young,mathrsfs,hyperref,bbm}
\usepackage[noend]{algpseudocode}
\usepackage[enableskew,vcentermath]{youngtab}
\usepackage[all]{xy}
\usetikzlibrary{decorations.pathreplacing,patterns}
\theoremstyle{plain}
\newtheorem{theorem}{Theorem}[section]

\newtheorem{corollary}[theorem]{Corollary}
\newtheorem{lemma}[theorem]{Lemma}

\theoremstyle{definition}
\newtheorem{definition}[theorem]{Definition}
\newtheorem{algorithm}[theorem]{Algorithm}

\newcommand{\rr}{\mathbb{R}}

\newcommand{\cc}{\mathbb{C}}

\newcommand{\calM}{\mathcal{M}}
\newcommand{\calN}{\mathcal{N}}

\newcommand{\calP}{\mathcal{P}}

\newcommand{\Le}{\reflectbox{L}}

\title{Positroids have the Rayleigh property}
\author{Cameron Marcott}
\date{}

\begin{document}
\maketitle

\begin{abstract}
{\bf{Abstract:}} The Rayleigh property is a negative correlation property of electrical networks, which was generalized to matroids by Choe and Wagner. We prove that positroids, a class of matroid introduced by Postnikov which have seen many recent applications in both math and physics, have the Rayleigh property.
\end{abstract}

\section{Introduction}

Postnikov introduced positroids, matroids representable by a point in the Grassmannian with all positive Pl\"ucker coordinates, in \cite{postnikov:positive_grassmannian}. While the collection of all points representing a given matroid in the Grassmannian can be arbitrarily singular, Postnikov showed that the set of points in the positive Grassmannian representing a given positroid is homeomorphic to a ball. Positroids have seen many recent applications; a partial list includes Schubert calculus \cite{knutson:positroid}, total positivity \cite{postnikov:positive_grassmannian}, cluster algebras \cite{kodama:kp_solitons}, the inverse boundary problem for electrical networks \cite{postnikov:positive_grassmannian}, scattering amplitudes \cite{arkani:scattering}, and KP solitons \cite{kodama:kp_solitons}.

The Rayleigh monotonicity property of electrical networks says that if the conductance of any wire in a network is increased, then the flow of the network between any two nodes is not decreased. This property is equivalent to the events that an edge $e$ is in a random spanning tree and the event that an edge $f$ is in a random spanning tree are negatively correlated. Choe and Wagner introduced the matroidal Rayleigh property in \cite{choe:rayleigh_matroids}, generalizing a physical property of electrical networks. The Rayleigh property has been applied to study random walks and reversible Markov chains \cite{doyle:random_walks}. A negative correlation property equivalent to the Rayleigh property was studied by Pemantle \cite{pemantle:negative_dependence}. This negative correlation property has applications in matroid theory \cite{choe:rayleigh_matroids, wagner:masons_conjecture}, to log-concavity results \cite{kahn:negative_correlation}, and to approximation algorithms \cite{dubhashi:positive_influence, srinivasan:approximation_algorithms}. A strengthening of this negative correlation properties, the strong Rayleigh property, has connections to the theory of stable polynomials \cite{borcea:negative_dependence, branden:polynomials}.

Our result is that positroids exhibit the Rayleigh property.

\begin{theorem} \label{thm:rayleigh}
Positroids are Rayleigh matroids.
\end{theorem}

This theorem strengthens physical connections of positroids. Rayleigh matroids exhibit many combinatorial properties enjoyed by the electrical networks the Rayleigh property is abstracted from \cite{wagner:electrical_network}. So, Theorem \ref{thm:rayleigh} can be interpreted as saying that positroids might be expected to behave in much the same way of electrical networks. In matroid theory, the problem of describing exactly which matroids posses the Rayleigh property is of interest because Rayleigh matroids represent a large and tractable subclass of gammoids. Theorem \ref{thm:rayleigh} fits nicely into this research program by expanding the set of matroids known to be Rayleigh matroids. Theorem \ref{thm:rayleigh} also provides evidence positroids might enjoy the strong Rayleigh property, which would provide a suite of new tools to study positroids through the theory of stable polynomials \cite{branden:polynomials}. Finally, the connections between Rayleigh properties and probability theory can be used to prove results about fast mixing in the basis exchange graph of a matroid \cite{cohen:fast_mixing}. That is, if one starts from a fixed basis of the matroid and performs just a small number random basis exchanges, they can expect the probability distribution for basis they end up at to be a random distribution. Though we don't pursue the idea here, it would be interesting to interpret what this property says about the cluster algebra associated to a positroid variety.

Lattice path matroids, a subclass of positroids, were shown to posses the Rayleigh property independently by Cohen, Tetali, and Yeliussizov \cite{cohen:fast_mixing} and Xu \cite{xu:thesis}. Geometrically, lattice path matroids are the matroids represented by a generic point in a Richardson variety in a Grassmannian while positroids, are the matroids represented by a generic point in a positroid variety. Knutson, Lam, and Speyer showed in \cite{knutson:positroid} that positroid varieties enjoy many nice properties that Richardson varieties enjoy. This project was undertaken partially as a proof of concept of the combinatorial version of this claim: positroids should enjoy nice properties enjoyed by lattice path matroids. The arguments in \cite{cohen:fast_mixing} and \cite{xu:thesis} do not directly generalize to positroids. These arguments rely on the linear order of the ground set of a lattice path matroid, and fail for the cyclically ordered ground set of a positroid. These arguments both feature a seemingly unavoidable case analysis, which becomes unnecessary in the proof of Theorem \ref{thm:rayleigh}.

Section  \ref{sec:background} provides relevant background on matroids, positroids, and the Rayleigh property. Section \ref{sec:proof} is devoted to proving Theorem \ref{thm:rayleigh}. Rayleigh matroids are closely related to several other classes of matroids coming from strengthening or weakening the Rayleigh property as a negative correlation property. Section \ref{sec:other_classes} briefly discuses some of these classes and whether or not positroids belong to them.

\section{Background} \label{sec:background}

While reading this paper without prior exposure to matroids or positroids should be possible, the uninitiated reader's time would be better spent becoming familiar with these subjects. For introductions to matroids and positroids, we recommend \cite{borovik:coxeter_matroids} and \cite{postnikov:positive_grassmannian} respectively. Our notation concerning matroids and their basis enumerator polynomials is taken from \cite{choe:rayleigh_matroids}. Our perspective on positroids is taken from \cite{ardila:positroids}.

\subsection{Matroids} \label{sec:matroids}

All matroids considered in this paper are on the ground set $[n] = \{1,2,\dots, n\}$. We denote matroids by caligraphic letters and present them as their set of bases. So, for any rank $r$ matroid $\calM$, $\calM \subseteq \binom{[n]}{r}$. To save on ink, we omit curly brackets and commas when writing subsets. So, rather than $\calM = \{ \{1,2\}, \{1,3\},\{2,3\}\}$, we write $\calM = \{ 12,13,23\}$.

For disjoint subsets $I, J \subseteq [n]$, define
\begin{displaymath}
\calM^{J}_{I} = \{ B \in \calM : I \subseteq B, J \cap B = \emptyset \}.
\end{displaymath}
\noindent
When this set is not is empty, it is the minor obtained by contracting the set $I$ and deleting the set $J$ from $\calM$, commonly denoted $\calM /I \setminus J$. In cases where $\calM \setminus J \neq \emptyset$, but has strictly lower rank than $\calM$, $\calM^J = \emptyset \neq \calM \setminus J$.

Let $\mathbf{x} = \{x_1, x_2, \dots, x_n\}$ be a set of indeterminates. For a matroid $\calM$, denote by $M(\mathbf{x})$ the basis enumerator polynomial of the matroid $\calM$. That is,
\begin{displaymath}
M(\mathbf{x}) =
\sum_{B \in \calM}
\mathbf{x}^B,
\end{displaymath}
\noindent
where $\mathbf{x}^B = \prod_{i \in B} x_i$. If a matroid is denoted by a caligraphic letter, its basis enumerator polynomial is denoted by the corresponding roman letter. Observe that for $e, f \in [n]$, the basis enumerator polynomial for $\calM^{f}_{e}$ is
\begin{displaymath}
M^{f}_{e}(\mathbf{x}) =
\left. x_e \frac{\partial}{\partial x_e} M(\mathbf{x})\right|_{x_f = 0}.
\end{displaymath}

\subsection{Positroids} \label{sec:positroids}

Positroids are in bijection with several interesting combinatorial objects including \Le-diagrams, decorated permutations, Grassmann necklaces, and plablic graphs \cite{postnikov:positive_grassmannian}. We find it convenient to define positroids in terms their bases, which may be read off of a \Le-diagram using Proposition 6.2 of \cite{ardila:positroids}.

\begin{definition}
A \Le{\it{-diagram}} (or Le-diagram) is:
\begin{itemize}
\item a lattice path from $(0,0)$ to $(n-r,r)$, together with
\item a filling of the boxes in the Ferrers shape lying above the lattice path with dots such that if a box has a dot above it in the same column and to the left of it in the same row, that box also contains a dot.
\end{itemize}
\end{definition}

Figure \ref{fig:diagram_example} provides an example and non-example of a \Le-diagram.

\begin{figure}
\begin{displaymath}
\begin{tikzpicture}
\draw[thin,gray] (0,0) -- (4,0) -- (4,3) -- (0,3) -- (0,0);
\draw[thin,gray] (1,0) -- (1,3);
\draw[thin,gray] (2,1) -- (2,3);
\draw[thin,gray] (0,1) -- (2,1);
\draw[thin,gray] (0,2) -- (3,2);
\draw[line width=2pt] (0,0) -- (2,0) -- (2,1) -- (3,1) -- (3,3) -- (4,3);
\draw[fill=black] (.5,1.5) circle (.1);
\draw[fill=black] (1.5,1.5) circle (.1);
\draw[fill=black] (2.5,1.5) circle (.1);
\draw[fill=black] (1.5,2.5) circle (.1);
\draw[fill=black] (1.5,.5) circle (.1);

\begin{scope}[shift={(6,0)}]
\draw[thin,gray] (0,0) -- (4,0) -- (4,3) -- (0,3) -- (0,0);
\draw[thin,gray] (1,0) -- (1,3);
\draw[thin,gray] (2,1) -- (2,3);
\draw[thin,gray] (0,1) -- (2,1);
\draw[thin,gray] (0,2) -- (3,2);
\draw[line width=2pt] (0,0) -- (2,0) -- (2,1) -- (3,1) -- (3,3) -- (4,3);
\draw[fill=black] (.5,1.5) circle (.1);
\draw[fill=black] (2.5,1.5) circle (.1);
\draw[fill=black] (1.5,2.5) circle (.1);
\draw[fill=black] (1.5,.5) circle (.1);
\end{scope}
\end{tikzpicture}
\end{displaymath}
\caption{The figure on the left is a \protect\Le-diagram. The figure on the right is not.}
\label{fig:diagram_example}
\end{figure}
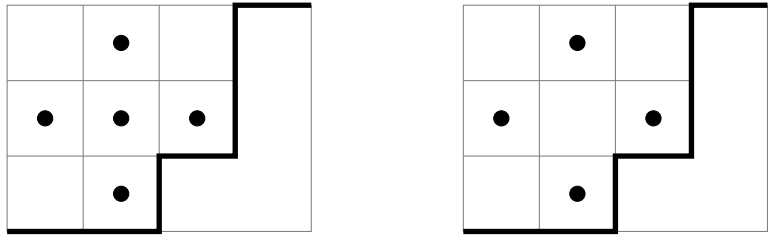

\begin{definition}
Given a \Le-digram, we construct a \Le{\it{-graph}} (or Le-graph) by:
\begin{itemize}
\item placing a node in the middle of each edge of the diagram's boundary path,
\item labelling the nodes on the boundary path $1$ to $n$ starting at the Northeast corner of the diagram and ending at the Southwest corner,
\item adding an edge directed to the left between any two nodes which lie in the same row and which have no other vertices between them in the same row, and
\item adding an edge directed downward between any two nodes which lie in the same column and which have no other vertices between them in the same column.
\end{itemize}
Figure \ref{fig:le_graph} provides an example of a \Le-graph.
\end{definition}

\begin{figure}
\begin{displaymath}
\begin{tikzpicture}
\draw[thin,gray] (0,0) -- (4,0) -- (4,3) -- (0,3) -- (0,0);
\draw[thin,gray] (1,0) -- (1,3);
\draw[thin,gray] (2,1) -- (2,3);
\draw[thin,gray] (0,1) -- (2,1);
\draw[thin,gray] (0,2) -- (3,2);
\draw[line width=2pt] (0,0) -- (2,0) -- (2,1) -- (3,1) -- (3,3) -- (4,3);
\draw[fill=black] (.5,1.5) circle (.1);
\draw[fill=black] (1.5,1.5) circle (.1);
\draw[fill=black] (2.5,1.5) circle (.1);
\draw[fill=black] (1.5,2.5) circle (.1);
\draw[fill=black] (1.5,.5) circle (.1);

\draw[fill=black] (.5,0) circle (.1);
\draw (.5,-.25) node {\scriptsize 7};
\draw[fill=black] (1.5,0) circle (.1);
\draw (1.5,-.25) node {\scriptsize 6};
\draw[fill=black] (2,.5) circle (.1);
\draw (2.25,.5) node {\scriptsize 5};
\draw[fill=black] (2.5,1) circle (.1);
\draw (2.5,.75) node {\scriptsize 4};
\draw[fill=black] (3,1.5) circle (.1);
\draw (3.25,1.5) node {\scriptsize 3};
\draw[fill=black] (3,2.5) circle (.1);
\draw (3.25,2.5) node {\scriptsize 2};
\draw[fill=black] (3.5,3) circle (.1);
\draw (3.5,2.75) node {\scriptsize 1};

\draw[->,line width=.75pt] (3,2.5) -- (1.6,2.5);
\draw[->,line width=.75pt] (3,1.5) -- (2.6,1.5);
\draw[->,line width=.75pt] (2.5,1.5) -- (1.6,1.5);
\draw[->,line width=.75pt] (1.5,1.5) -- (.6,1.5);
\draw[->,line width=.75pt] (2,.5) -- (1.6,.5);

\draw[->,line width=.75pt] (1.5,2.5) -- (1.5,1.6);
\draw[->,line width=.75pt] (1.5,1.5) -- (1.5,.6);
\draw[->,line width=.75pt] (1.5,.5) -- (1.5,.1);
\draw[->,line width=.75pt] (.5,1.5) -- (.5,.1);
\draw[->,line width=.75pt] (2.5,1.5) -- (2.5,1.1);
\end{tikzpicture}
\end{displaymath}
\caption{The \protect\Le-graph built from the \protect\Le-diagram in Figure \ref{fig:diagram_example}.}
\label{fig:le_graph}
\end{figure}
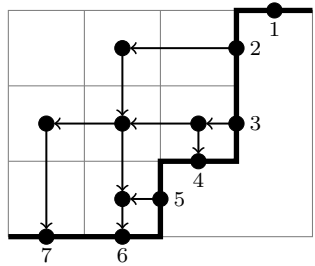

The following characterization of positroids in terms of their basis is given in \cite{ardila:positroids}. It is equivalent to all of the definitions of positroids found in \cite{postnikov:positive_grassmannian}.

\begin{definition} \label{def:positroid}
Let $B$ be the set of nodes labelling vertical steps of the boundary path of a \Le-graph. Let $\calP \subseteq \binom{[n]}{r}$ consist of $B$ together all sets $I \in \binom{[n]}{r}$ such that there exists an edge disjoint walk from $B \setminus I$ to $I \setminus B$ in the \Le-graph. This collection of subsets is a {\it{positroid}}, which is a type of matroid. All positroids may be realized uniquely in this fashion.
\end{definition}

For example, the positroid built from the \Le-graph in Figure \ref{fig:le_graph} is
\begin{displaymath}
\calP = \{235, 236, 245, 246, 256, 257, 267, 356, 357, 367, 456, 457, 467\}.
\end{displaymath}

\subsection{The Rayleigh property} \label{sec:rayleigh}

A graph with positive weighted edges may be viewed as an electrical network with the edge weights corresponding to electrical conductances of a wire. Given any two vertices $a,b$ of the graph, one may measure the conductance of the network as a whole from $a$ to $b$. Rayleigh's monotonicity law says that if the conductance of any edge in the network is increased, the conductance of the network from $a$ to $b$ will not decrease. Choe and Wagner generalized this physical property to matroids in \cite{choe:rayleigh_matroids}.

\begin{definition}
A matroid $\calM$ has the {\it{Rayleigh property}} if for all $e,f \in [n]$ and all $\mathbf{a} = (a_1, a_2, \dots, a_n) \in \rr_{\geq 0}^{n}$,
\begin{equation} \label{eqn:rayleigh}
M_{ef}(\mathbf{a}) M^{ef}(\mathbf{a})
\leq M^{f}_{e}(\mathbf{a}) M^{e}_{f}(\mathbf{a}).
\end{equation}
\end{definition}

Viewing the positive weight vector $\mathbf{a}$ as a probability distribution on the ground set $[n]$, we may view the polynomial $M(\mathbf{a})$ as giving a probability distribution on the bases of the matroid $\calM$. Under this interpretation, the inequality (\ref{eqn:rayleigh}) says that no matter which probability distribution is placed on the ground set, the events that $e$ is in a randomly chosen basis and that $f$ is in a randomly chosen basis are anti-correlated.

The Rayleigh property is related to several other matroidal properties, some of which are mentioned in Section \ref{sec:other_classes}.

\section{Proof of Theorem \ref{thm:rayleigh}} \label{sec:proof}

To prove that (\ref{eqn:rayleigh}) holds for a positroid $\calP$, it suffices to find for each $e,f \in [n]$ an injection
\begin{displaymath}
\phi: \calP_{ef} \times \calP^{ef} \to \calP^{f}_{e} \times \calP^{e}_{f}
\end{displaymath}
\noindent
such that for all $(B_1,B_2) \in \calP_{ef} \times \calP^{ef}$, $\mathbf{x}^{B_1} \mathbf{x}^{B_2} = \mathbf{x}^{B'_1} \mathbf{x}^{B'_2}$, where $(B'_1,B'_2) = \phi(B_1,B_2)$. This section describes an algorithm giving such an injection.

An element $(B_1,B_2) \in \calP_{ef} \times \calP^{ef}$ may be represented as two sets of vertex disjoint paths in the \Le-graph defining $\calP$. Note that while each set of paths is disjoint, it is possible paths from the first set intersect paths from the second. We color the collection of paths representing $B_1$ blue and the collection representing $B_2$ green. Suppose $e$ labels a horizontal edge of the boundary of the \Le-digaram. Then, since $e \in B_1$ and $e \notin B_2$, there is a blue path, but not a green path terminating at $e$. Likewise, if $e$ labels a vertical edge of the boundary of the \Le-diagram, there is a green path, but not a blue path originating at $e$. Similar statements may be made about $f$.

Algorithm \ref{alg:injection} works by placing a marker on the vertex $f$. This marker walks along the edges of the \Le-graph in a prescribed fashion, switching the colors of edges as it traverses them. The algorithm returns a new coloring of the edges of the \Le-graph, which represents a new pair of bases of the positroid.

\begin{algorithm} \label{alg:injection}
Let $\calP$ positroid presented as a \Le-diagram, for each $B \in \calP$ choose a collection of vertex disjoint paths $P$ in the \Le-graph representing $B$.

\begin{algorithmic}
\State {\bf{input:}} $e,f \in [n]$, $(B_1,B_2) \in \calP_{ef} \times \calP^{ef}$. \vskip.1cm
\State Color the collection of paths representing $B_1$ blue. 
\State Color the collection of paths representing $B_2$ green.
\State Place a marker on the vertex labeled $f$.
\State Color the marker the same color as the unique colored path incident to $f$. \vskip.1cm
 
\Loop \vskip.1cm
\If {The marker is blue,} move the marker along the blue edge pointing into the marker. If there is more than one such edge, use the edge the marker did not use to enter the node it is located. \vskip.1cm

\If {The marker is incident to both a blue path and a green path,} color the marker green.
\EndIf \vskip.1cm

\State Recolor the traversed edge green. \vskip.1cm

\If {The marker is on the boundary of the \Le-diagram,} {\bf{stop}}.
\EndIf \vskip.1cm

\EndIf \vskip.1cm

\If {The marker is green,} move the marker along the green edge pointing away from the marker. If there is more than one such edge, use the edge the marker did not use to enter the node it is located. \vskip.1cm

\If {The marker is incident to both a blue path and a green path,} color the marker blue. \vskip.1cm
\EndIf

\State Recolor the traversed edge blue. \vskip.1cm

\If {The marker is on the boundary of the \Le-diagram,} {\bf{stop}}.
\EndIf
\EndIf
\EndLoop
\vskip.1cm
\noindent
{\bf{end loop}}
\vskip.1cm
\State {\bf{output:}} $(B'_1,B'_2) \in \calP^{f}_{e} \times \calP^{e}_{f}$, where $B'_1$ is given by the new collection of blue paths and $B'_2$ is given by the new collection of green paths.
\end{algorithmic}
\end{algorithm}

Figure \ref{fig:algorithm} provides an example of Algorithm \ref{alg:injection}.

\begin{figure}[h]
\begin{displaymath}
\begin{tikzpicture}[scale=.95]
\draw[thin,gray] (0,0) -- (4,0) -- (4,3) -- (0,3) -- (0,0);
\draw[thin,gray] (1,0) -- (1,3);
\draw[thin,gray] (2,1) -- (2,3);
\draw[thin,gray] (0,1) -- (2,1);
\draw[thin,gray] (0,2) -- (3,2);
\draw[line width=2pt] (0,0) -- (2,0) -- (2,1) -- (3,1) -- (3,3) -- (4,3);
\draw[fill=black] (.5,1.5) circle (.1);
\draw[fill=black] (1.5,1.5) circle (.1);
\draw[fill=black] (2.5,1.5) circle (.1);
\draw[fill=black] (1.5,2.5) circle (.1);
\draw[fill=black] (1.5,.5) circle (.1);

\draw[blue,fill=blue] (.5,0) circle (.2);
\draw (.75,-.25) node {\scriptsize 7};
\draw[fill=black] (1.5,0) circle (.1);
\draw (1.5,-.25) node {\scriptsize 6};
\draw[fill=black] (2,.5) circle (.1);
\draw (2.25,.5) node {\scriptsize 5};
\draw[fill=black] (2.5,1) circle (.1);
\draw (2.5,.75) node {\scriptsize 4};
\draw[fill=black] (3,1.5) circle (.1);
\draw (3.25,1.5) node {\scriptsize 3};
\draw[fill=black] (3,2.5) circle (.1);
\draw (3.25,2.5) node {\scriptsize 2};
\draw[fill=black] (3.5,3) circle (.1);
\draw (3.5,2.75) node {\scriptsize 1};

\draw[->,line width=.75pt] (3,2.5) -- (1.6,2.5);
\draw[->,line width=.75pt] (3,1.5) -- (2.6,1.5);
\draw[->,line width=.75pt] (2.5,1.5) -- (1.6,1.5);
\draw[->,line width=.75pt] (1.5,1.5) -- (.6,1.5);
\draw[->,line width=.75pt] (2,.5) -- (1.6,.5);

\draw[->,line width=.75pt] (1.5,2.5) -- (1.5,1.6);
\draw[->,line width=.75pt] (1.5,1.5) -- (1.5,.6);
\draw[->,line width=.75pt] (1.5,.5) -- (1.5,.1);
\draw[->,line width=.75pt] (.5,1.5) -- (.5,.2);
\draw[->,line width=.75pt] (2.5,1.5) -- (2.5,1.1);

\draw[blue, dash pattern = on 1pt off 1.5pt, line width = 7pt] (.5,.3) -- (.5,1.4);
\draw[blue, dash pattern = on 1pt off 1.5pt, line width = 7pt] (.7,1.5) -- (1.4,1.5);
\draw[blue, dash pattern = on 1pt off 1.5pt, line width = 7pt] (1.7,1.5) -- (2.4,1.5);
\draw[blue, dash pattern = on 1pt off 1.5pt, line width = 7pt] (2.7,1.5) -- (2.9,1.5);

\draw[blue, dash pattern = on 1pt off 1.5pt, line width = 7pt] (1.5,.2) -- (1.5,.4);
\draw[blue, dash pattern = on 1pt off 1.5pt, line width = 7pt] (1.7,.5) -- (1.9,.5);

\draw[green, dash pattern = on 1pt off 1.5pt, line width = 7pt, dash phase = 1.25pt] (1.5,.2) -- (1.5,.4);
\draw[green, dash pattern = on 1pt off 1.5pt, line width = 7pt] (1.5,.7) -- (1.5,1.4);
\draw[green, dash pattern = on 1pt off 1.5pt, line width = 7pt] (1.5,1.7) -- (1.5,2.4);
\draw[green, dash pattern = on 1pt off 1.5pt, line width = 7pt] (1.7,2.5) -- (2.9,2.5);

\begin{scope}[shift={(6,0)}]
\draw[thin,gray] (0,0) -- (4,0) -- (4,3) -- (0,3) -- (0,0);
\draw[thin,gray] (1,0) -- (1,3);
\draw[thin,gray] (2,1) -- (2,3);
\draw[thin,gray] (0,1) -- (2,1);
\draw[thin,gray] (0,2) -- (3,2);
\draw[line width=2pt] (0,0) -- (2,0) -- (2,1) -- (3,1) -- (3,3) -- (4,3);
\draw[fill=black] (1.5,1.5) circle (.1);
\draw[fill=black] (2.5,1.5) circle (.1);
\draw[fill=black] (1.5,2.5) circle (.1);
\draw[fill=black] (1.5,.5) circle (.1);

\draw[fill=black] (.5,0) circle (.1);
\draw (.5,-.25) node {\scriptsize 7};
\draw[fill=black] (1.5,0) circle (.1);
\draw (1.5,-.25) node {\scriptsize 6};
\draw[fill=black] (2,.5) circle (.1);
\draw (2.25,.5) node {\scriptsize 5};
\draw[fill=black] (2.5,1) circle (.1);
\draw (2.5,.75) node {\scriptsize 4};
\draw[fill=black] (3,1.5) circle (.1);
\draw (3.25,1.5) node {\scriptsize 3};
\draw[fill=black] (3,2.5) circle (.1);
\draw (3.25,2.5) node {\scriptsize 2};
\draw[fill=black] (3.5,3) circle (.1);
\draw (3.5,2.75) node {\scriptsize 1};

\draw[->,line width=.75pt] (3,2.5) -- (1.6,2.5);
\draw[->,line width=.75pt] (3,1.5) -- (2.6,1.5);
\draw[->,line width=.75pt] (2.5,1.5) -- (1.6,1.5);
\draw[->,line width=.75pt] (1.5,1.5) -- (.7,1.5);
\draw[->,line width=.75pt] (2,.5) -- (1.6,.5);

\draw[->,line width=.75pt] (1.5,2.5) -- (1.5,1.6);
\draw[->,line width=.75pt] (1.5,1.5) -- (1.5,.6);
\draw[->,line width=.75pt] (1.5,.5) -- (1.5,.1);
\draw[->,line width=.75pt] (.5,1.4) -- (.5,.1);
\draw[->,line width=.75pt] (2.5,1.5) -- (2.5,1.1);

\draw[green, dash pattern = on 1pt off 1.5pt, line width = 7pt] (.5,.2) -- (.5,1.3);
\draw[blue, dash pattern = on 1pt off 1.5pt, line width = 7pt] (.8,1.5) -- (1.4,1.5);
\draw[blue, dash pattern = on 1pt off 1.5pt, line width = 7pt] (1.7,1.5) -- (2.4,1.5);
\draw[blue, dash pattern = on 1pt off 1.5pt, line width = 7pt] (2.7,1.5) -- (2.9,1.5);

\draw[blue, dash pattern = on 1pt off 1.5pt, line width = 7pt] (1.5,.2) -- (1.5,.4);
\draw[blue, dash pattern = on 1pt off 1.5pt, line width = 7pt] (1.7,.5) -- (1.9,.5);

\draw[green, dash pattern = on 1pt off 1.5pt, line width = 7pt, dash phase = 1.25pt] (1.5,.2) -- (1.5,.4);
\draw[green, dash pattern = on 1pt off 1.5pt, line width = 7pt] (1.5,.7) -- (1.5,1.4);
\draw[green, dash pattern = on 1pt off 1.5pt, line width = 7pt] (1.5,1.7) -- (1.5,2.4);
\draw[green, dash pattern = on 1pt off 1.5pt, line width = 7pt] (1.7,2.5) -- (2.9,2.5);
\draw[blue,fill=blue]  (.5,1.5) circle (.2);
\end{scope}

\begin{scope}[shift={(12,0)}]
\draw[thin,gray] (0,0) -- (4,0) -- (4,3) -- (0,3) -- (0,0);
\draw[thin,gray] (1,0) -- (1,3);
\draw[thin,gray] (2,1) -- (2,3);
\draw[thin,gray] (0,1) -- (2,1);
\draw[thin,gray] (0,2) -- (3,2);
\draw[line width=2pt] (0,0) -- (2,0) -- (2,1) -- (3,1) -- (3,3) -- (4,3);
\draw[fill=black] (.5,1.5) circle (.1);
\draw[fill=black] (2.5,1.5) circle (.1);
\draw[fill=black] (1.5,2.5) circle (.1);
\draw[fill=black] (1.5,.5) circle (.1);

\draw[fill=black] (.5,0) circle (.1);
\draw (.5,-.25) node {\scriptsize 7};
\draw[fill=black] (1.5,0) circle (.1);
\draw (1.5,-.25) node {\scriptsize 6};
\draw[fill=black] (2,.5) circle (.1);
\draw (2.25,.5) node {\scriptsize 5};
\draw[fill=black] (2.5,1) circle (.1);
\draw (2.5,.75) node {\scriptsize 4};
\draw[fill=black] (3,1.5) circle (.1);
\draw (3.25,1.5) node {\scriptsize 3};
\draw[fill=black] (3,2.5) circle (.1);
\draw (3.25,2.5) node {\scriptsize 2};
\draw[fill=black] (3.5,3) circle (.1);
\draw (3.5,2.75) node {\scriptsize 1};

\draw[->,line width=.75pt] (3,2.5) -- (1.6,2.5);
\draw[->,line width=.75pt] (3,1.5) -- (2.6,1.5);
\draw[->,line width=.75pt] (2.5,1.5) -- (1.7,1.5);
\draw[->,line width=.75pt] (1.4,1.5) -- (.6,1.5);
\draw[->,line width=.75pt] (2,.5) -- (1.6,.5);

\draw[->,line width=.75pt] (1.5,2.5) -- (1.5,1.7);
\draw[->,line width=.75pt] (1.5,1.4) -- (1.5,.6);
\draw[->,line width=.75pt] (1.5,.5) -- (1.5,.1);
\draw[->,line width=.75pt] (.5,1.5) -- (.5,.1);
\draw[->,line width=.75pt] (2.5,1.5) -- (2.5,1.1);

\draw[green, dash pattern = on 1pt off 1.5pt, line width = 7pt] (.5,.2) -- (.5,1.4);
\draw[green, dash pattern = on 1pt off 1.5pt, line width = 7pt] (.7,1.5) -- (1.3,1.5);
\draw[blue, dash pattern = on 1pt off 1.5pt, line width = 7pt] (1.8,1.5) -- (2.4,1.5);
\draw[blue, dash pattern = on 1pt off 1.5pt, line width = 7pt] (2.7,1.5) -- (2.9,1.5);

\draw[blue, dash pattern = on 1pt off 1.5pt, line width = 7pt] (1.5,.2) -- (1.5,.4);
\draw[blue, dash pattern = on 1pt off 1.5pt, line width = 7pt] (1.7,.5) -- (1.9,.5);

\draw[green, dash pattern = on 1pt off 1.5pt, line width = 7pt, dash phase = 1.25pt] (1.5,.2) -- (1.5,.4);
\draw[green, dash pattern = on 1pt off 1.5pt, line width = 7pt] (1.5,.7) -- (1.5,1.3);
\draw[green, dash pattern = on 1pt off 1.5pt, line width = 7pt] (1.5,1.8) -- (1.5,2.4);
\draw[green, dash pattern = on 1pt off 1.5pt, line width = 7pt] (1.7,2.5) -- (2.9,2.5);
\draw[green,fill=green]  (1.5,1.5) circle (.2);
\end{scope}

\begin{scope}[shift={(3,-4)}]
\draw[thin,gray] (0,0) -- (4,0) -- (4,3) -- (0,3) -- (0,0);
\draw[thin,gray] (1,0) -- (1,3);
\draw[thin,gray] (2,1) -- (2,3);
\draw[thin,gray] (0,1) -- (2,1);
\draw[thin,gray] (0,2) -- (3,2);
\draw[line width=2pt] (0,0) -- (2,0) -- (2,1) -- (3,1) -- (3,3) -- (4,3);
\draw[fill=black] (.5,1.5) circle (.1);
\draw[fill=black] (1.5,1.5) circle (.1);
\draw[fill=black] (2.5,1.5) circle (.1);
\draw[fill=black] (1.5,2.5) circle (.1);

\draw[fill=black] (.5,0) circle (.1);
\draw (.5,-.25) node {\scriptsize 7};
\draw[fill=black] (1.5,0) circle (.1);
\draw (1.5,-.25) node {\scriptsize 6};
\draw[fill=black] (2,.5) circle (.1);
\draw (2.25,.5) node {\scriptsize 5};
\draw[fill=black] (2.5,1) circle (.1);
\draw (2.5,.75) node {\scriptsize 4};
\draw[fill=black] (3,1.5) circle (.1);
\draw (3.25,1.5) node {\scriptsize 3};
\draw[fill=black] (3,2.5) circle (.1);
\draw (3.25,2.5) node {\scriptsize 2};
\draw[fill=black] (3.5,3) circle (.1);
\draw (3.5,2.75) node {\scriptsize 1};

\draw[->,line width=.75pt] (3,2.5) -- (1.6,2.5);
\draw[->,line width=.75pt] (3,1.5) -- (2.6,1.5);
\draw[->,line width=.75pt] (2.5,1.5) -- (1.6,1.5);
\draw[->,line width=.75pt] (1.5,1.5) -- (.6,1.5);
\draw[->,line width=.75pt] (2,.5) -- (1.7,.5);

\draw[->,line width=.75pt] (1.5,2.5) -- (1.5,1.6);
\draw[->,line width=.75pt] (1.5,1.5) -- (1.5,.7);
\draw[->,line width=.75pt] (1.5,.5) -- (1.5,.1);
\draw[->,line width=.75pt] (.5,1.5) -- (.5,.1);
\draw[->,line width=.75pt] (2.5,1.5) -- (2.5,1.1);

\draw[green, dash pattern = on 1pt off 1.5pt, line width = 7pt] (.5,.2) -- (.5,1.4);
\draw[green, dash pattern = on 1pt off 1.5pt, line width = 7pt] (.7,1.5) -- (1.4,1.5);
\draw[blue, dash pattern = on 1pt off 1.5pt, line width = 7pt] (1.7,1.5) -- (2.4,1.5);
\draw[blue, dash pattern = on 1pt off 1.5pt, line width = 7pt] (2.7,1.5) -- (2.9,1.5);

\draw[blue, dash pattern = on 1pt off 1.5pt, line width = 7pt] (1.5,.2) -- (1.5,.3);
\draw[blue, dash pattern = on 1pt off 1.5pt, line width = 7pt] (1.8,.5) -- (1.9,.5);

\draw[green, dash pattern = on 1pt off 1.5pt, line width = 7pt, dash phase = 1.25pt] (1.5,.2) -- (1.5,.3);
\draw[blue, dash pattern = on 1pt off 1.5pt, line width = 7pt] (1.5,.8) -- (1.5,1.4);
\draw[green, dash pattern = on 1pt off 1.5pt, line width = 7pt] (1.5,1.7) -- (1.5,2.4);
\draw[green, dash pattern = on 1pt off 1.5pt, line width = 7pt] (1.7,2.5) -- (2.9,2.5);
\draw[blue,fill=blue]  (1.5,.5) circle (.2);
\end{scope}

\begin{scope}[shift={(9,-4)}]
\draw[thin,gray] (0,0) -- (4,0) -- (4,3) -- (0,3) -- (0,0);
\draw[thin,gray] (1,0) -- (1,3);
\draw[thin,gray] (2,1) -- (2,3);
\draw[thin,gray] (0,1) -- (2,1);
\draw[thin,gray] (0,2) -- (3,2);
\draw[line width=2pt] (0,0) -- (2,0) -- (2,1) -- (3,1) -- (3,3) -- (4,3);
\draw[fill=black] (.5,1.5) circle (.1);
\draw[fill=black] (1.5,1.5) circle (.1);
\draw[fill=black] (2.5,1.5) circle (.1);
\draw[fill=black] (1.5,2.5) circle (.1);
\draw[fill=black] (1.5,.5) circle (.1);

\draw[fill=black] (.5,0) circle (.1);
\draw (.5,-.25) node {\scriptsize 7};
\draw[fill=black] (1.5,0) circle (.1);
\draw (1.5,-.25) node {\scriptsize 6};
\draw[fill=black] (2,.5) circle (.1);
\draw (2.25,.25) node {\scriptsize 5};
\draw[fill=black] (2.5,1) circle (.1);
\draw (2.5,.75) node {\scriptsize 4};
\draw[fill=black] (3,1.5) circle (.1);
\draw (3.25,1.5) node {\scriptsize 3};
\draw[fill=black] (3,2.5) circle (.1);
\draw (3.25,2.5) node {\scriptsize 2};
\draw[fill=black] (3.5,3) circle (.1);
\draw (3.5,2.75) node {\scriptsize 1};

\draw[->,line width=.75pt] (3,2.5) -- (1.6,2.5);
\draw[->,line width=.75pt] (3,1.5) -- (2.6,1.5);
\draw[->,line width=.75pt] (2.5,1.5) -- (1.6,1.5);
\draw[->,line width=.75pt] (1.5,1.5) -- (.6,1.5);
\draw[->,line width=.75pt] (2,.5) -- (1.6,.5);

\draw[->,line width=.75pt] (1.5,2.5) -- (1.5,1.6);
\draw[->,line width=.75pt] (1.5,1.5) -- (1.5,.6);
\draw[->,line width=.75pt] (1.5,.5) -- (1.5,.1);
\draw[->,line width=.75pt] (.5,1.5) -- (.5,.1);
\draw[->,line width=.75pt] (2.5,1.5) -- (2.5,1.1);

\draw[green, dash pattern = on 1pt off 1.5pt, line width = 7pt] (.5,.2) -- (.5,1.4);
\draw[green, dash pattern = on 1pt off 1.5pt, line width = 7pt] (.7,1.5) -- (1.4,1.5);
\draw[blue, dash pattern = on 1pt off 1.5pt, line width = 7pt] (1.7,1.5) -- (2.4,1.5);
\draw[blue, dash pattern = on 1pt off 1.5pt, line width = 7pt] (2.7,1.5) -- (2.9,1.5);

\draw[blue, dash pattern = on 1pt off 1.5pt, line width = 7pt] (1.5,.2) -- (1.5,.4);
\draw[green, dash pattern = on 1pt off 1.5pt, line width = 7pt] (1.7,.5) -- (1.8,.5);

\draw[green, dash pattern = on 1pt off 1.5pt, line width = 7pt, dash phase = 1.25pt] (1.5,.2) -- (1.5,.4);
\draw[blue, dash pattern = on 1pt off 1.5pt, line width = 7pt] (1.5,.7) -- (1.5,1.4);
\draw[green, dash pattern = on 1pt off 1.5pt, line width = 7pt] (1.5,1.7) -- (1.5,2.4);
\draw[green, dash pattern = on 1pt off 1.5pt, line width = 7pt] (1.7,2.5) -- (2.9,2.5);
\draw[blue,fill=blue]  (2,.5) circle (.2);
\end{scope}

\draw[->,line width=1pt] (4.5,1.5) -- (5.5,1.5);
\draw[->,line width=1pt] (10.5,1.5) -- (11.5,1.5);
\draw[->,line width=1pt] (1.5,-2.5) -- (2.5,-2.5);
\draw[->,line width=1pt] (7.5,-2.5) -- (8.5,-2.5);
\end{tikzpicture}
\end{displaymath}
\caption{An example of Algorithm \ref{alg:injection}. In this example, $e = 2$, $f =7$, and $(B_1,B_2) = (267,356)$. The algorithm returns the pair $(256,367)$.}
\label{fig:algorithm}
\end{figure}
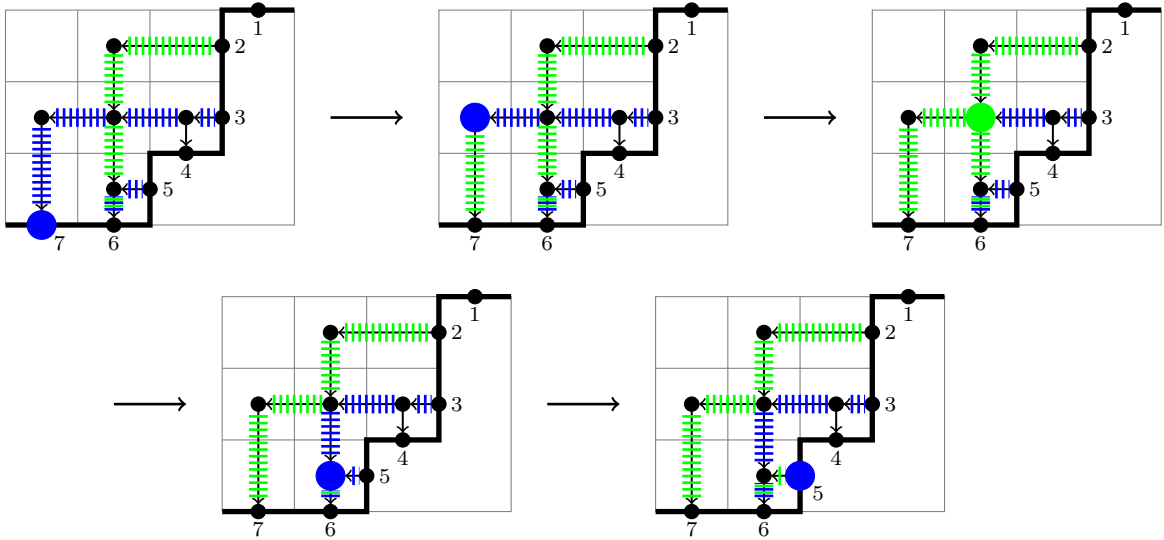

\begin{lemma} \label{lem:reversible}
Every step of Algorithm \ref{alg:injection} is reversible. 
\end{lemma}
\begin{proof}
To reverse Algorithm \ref{alg:injection}, preform the exact same procedure (placing a marker on node $f$, coloring it the same color as the edge incident to $f$, etc.) only this time traverse green edges against their flow and blue edges with their flow. It is easily seen that this procedure reverses any individual step of Algorithm \ref{alg:injection}.
\end{proof}

While the reversing procedure of Lemma \ref{lem:reversible} can be applied to any element of $\calP^f_e \times \calP^e_f$ it will not always yield an element of $\calP_{ef} \times \calP^{ef}$. For instance, applying the reversing procedure to a configuration that contains a green path from $e$ to $f$ which does not cross any blue paths results in an element of  $\calP^e_f \times \calP^f_e$, not an element of $\calP_{ef} \times \calP^{ef}$.

\begin{lemma} \label{lem:not_f}
In Algorithm \ref{alg:injection}, the marker cannot enter the vertices $e$ or $f$. In particular, Algorithm \ref{alg:injection} terminates. 
\end{lemma}
\begin{proof}
Using the notation from Algorithm \ref{alg:injection}, let $(B_1,B_2) \in \calP_{ef} \times \calP^{ef}$ be represented a blue and green collection of vertex disjoint paths in a \Le-diagram. Suppose $e$ is on a horizontal edge of the boundary of the \Le-diagram. Then, the directed edge pointing into $e$ is blue. Since the marker only traverses blue edges against their flow, the marker will never enter the vertex $e$. If $e$ is on a vertical edge of the boundary of the \Le-diagram, the directed edge pointing into $e$ is green. Since the marker only traverses green edges in the direction of the edge, the marker will never enter the vertex $e$.

Suppose the marker enters the vertex $f$ after leaving it initially. If $f$ is on a horizontal edge, after entering $f$, $f$ will be incident to a single blue edge pointing into it. However, we cannot apply the reversing procedure of \ref{lem:reversible} to such a configuration, contradicting the fact that each step of Algorithm \ref{alg:injection} is reversible. We reach a similar contradiction if $f$ is on a vertical edge. Hence the marker cannot return to the vertex $f$.

Since each step of Algorithm \ref{alg:injection} is reversible, the algorithm must either terminate or the diagram must revisit its initial configuration. Since the marker cannot return to the vertex $f$, Algorithm \ref{alg:injection} terminates.
\end{proof}

\begin{lemma} \label{lem:the_right_map}
The result of Algorithm \ref{alg:injection} is two sets of vertex disjoint walks, representing an element of $\calP^f_e \times \calP^e_f$.
\end{lemma}
\begin{proof}
Observe that after the marker leaves any vertex of the diagram which is not on the boundary path, the number of blue edges pointing into the vertex is the same as the number of blue edges pointing out of the vertex and the number of green edges pointing into the vertex is the same as the number of green edges pointing out of the vertex. Hence, the result of Algorithm \ref{alg:injection} is a collection of directed paths and cycles. Since the \Le-graph contains no directed cycles, the result of the algorithm is a collection of directed paths. Since the algorithm cannot cause a vertex to have four blue or four green vertices incident to it, Algorithm \ref{alg:injection} returns a collection of disjoint blue paths and a collection of disjoint green paths representing a pair $(B'_1,B'_2) \in \calP \times \calP$. From Lemma  \ref{lem:not_f}, Algorithm \ref{alg:injection} changes the color of the edge incident to $f$, but not the color of the edge incident to $e$. So, $(B'_1,B'_2) \in \calP^f_e \times \calP^e_f$.
\end{proof}

\begin{proof}[Proof of Theorem \ref{thm:rayleigh}]
Lemma \ref{lem:the_right_map} shows that Algorithm \ref{alg:injection} maps a pair of bases of the positroid $(B_1, B_2) \in \calP_{ef} \times \calP^{ef}$ to a pair of bases $(B'_1, B'_2) \in \calP^{f}_{e} \times \calP^{e}_{f}$. Lemma \ref{lem:reversible} shows that this map is injective. It is apparent that as multisets $B_1 \cup B_2 = B'_1 \cup B'_2$. Hence, ${\bf{x}}^{B_1}{\bf{x}}^{B_2} = {\bf{x}}^{B'_1} {\bf{x}}^{B'_2}$. By the comments at the beginning of this section, the existence of such an injection proves Theorem \ref{thm:rayleigh}.
\end{proof}

\section{Related classes of matroids} \label{sec:other_classes}

The class of Rayleigh matroids is closely related to several other interesting classes of matroids. We mention some of these classes and whether or not positroids are known to live in these classes.

A matroid $\calM$ is called {\it{balanced}} if for every minor $\calN$ of $\calM$,
\begin{displaymath}
|\calN_{ef} \times \calN^{ef}| \leq |\calN_{e}^f \times \calN_f^e|,
\end{displaymath}
\noindent
for all $e,f \in [n]$. Balanced matroids were introduced by Feder and Mihail in \cite{feder:balanced} in relation to a conjecture about the one-skeletons $\{0,1\}$-polytopes and in part motivated the introduction of Rayleigh matroids. The class of Rayleigh matroids is a subclass of the class of balanced matroids, though the converse is not true. So, we obtain the immediate corollary to Theorem \ref{thm:rayleigh}.

\begin{corollary}
Positroids are balanced matroids.
\end{corollary}

A matroid is said to be {\it{strongly Rayleigh}} if the inequality (\ref{eqn:rayleigh}) holds for all real input (as opposed to just all positive real input). Br\"and\'en showed in \cite{branden:polynomials} that being strongly Rayleigh is equivalent to the polynomial $M(\mathbf{x})$ being stable, which means that $M(\mathbf{z}) \neq 0$ for all $\mathbf{z} = (z_1,z_2,\dots,z_n) \in \cc^n$ such that $\mathrm{Im}(z_i) > 0$ for all $i$. Hence, strongly Rayleigh matroids are also called {\it{half-plane property matroids}}. It is unknown whether positroids are strongly Rayleigh.



\begin{thebibliography}{}

\bibitem{arkani:scattering}
N. Arkani-Hamed, et al..
\newblock {\it{Scattering amplitudes and the positive Grassmannian}},
\newblock \href{https://arxiv.org/abs/1212.5605}{arXiv:1212.5605}

\bibitem{ardila:positroids}
F$.$ Ardila, F$.$ Rincon, and L$.$ Williams.
\newblock {\it{Positroids and non-crossing partitions}},
\newblock Transactions of the American Mathematical Society, 368 (2016), 337--363.

\bibitem{borcea:negative_dependence}
J$.$ Borcea, P$.$ Br\"and\'en, T$.$ Liggett.
\newblock {\it{Negative dependence and the geometry of polynomials}},
\newblock Journal of the American Mathematical Society, 22 (2009), 521--567.

\bibitem{branden:polynomials}
P$.$ Br\"and\'en.
\newblock {\it{Polynomials with the half-plane property and matroid theory}},
\newblock Advances in Mathematics, 216 (2007), 302--320.

\bibitem{borovik:coxeter_matroids}
A$.$ Borovik, I$.$ Gelfand, and N$.$ White.
\newblock {\it{Coxeter matroids}},
\newblock Progress in Mathematics, 216, Boston: Birkh\"auser (2003).

\bibitem{bonin:lattice_path_matroid}
J$.$ Bonin, A$.$ de Mier, and M$.$ Noy.
\newblock {\it{Lattice path matroids: enumerative aspects and Tutte polynomials}},
\newblock Journal of Combinatorial Theory Series A, 104 (2003), 63--94.

\bibitem{choe:homogeneous_polynomials}
Y$.$ Choe, J$.$ Oxley, A$.$ Sokal, and D$.$ Wagner.
\newblock {\it{Homogeneous multivariate polynomials with the half-plane property}},
\newblock Advances in Applied Mathematics, 32 (2004), 88--187.

\bibitem{choe:rayleigh_matroids}
Y$.$ Choe, and D$.$ Wagner.
\newblock {\it{Rayleigh matroids}},
\newblock Combinatorics, Probability and Computing, 15 (2006), 765--781.

\bibitem{cohen:fast_mixing}
E$.$ Cohen, P$.$ Tetali, and D$.$ Yeliussizov.
\newblock {\it{Lattice path matroids: negative correlation and fast mixing}},
\newblock \href{http://arxiv.org/abs/1505.06710}{arXiv:1505.06710}.

\bibitem{doyle:random_walks}
P$.$ Doyle, and J$.$ Snell.
\newblock {\it{Random walks and electric networks}},
\newblock \href{https://arxiv.org/abs/math/0001057}{arXiv:math/0001057}

\bibitem{dubhashi:positive_influence}
D$.$ Dubhashi, J$.$ Jonasson, and D$.$ Ranjan.
\newblock {\it{Positive influence and negative dependence}},
Combinatorics, Probability, and Computing, 16 (2007), 29--41.

\bibitem{feder:balanced}
T$.$ Feder, and M$.$ Mihail.
\newblock {\it{Balanced matroids}},
\newblock in ``Proceedings of the 24th Annual ACM (STOC)", Victoria B.C., ACM Press, New York, 1992.

\bibitem{kahn:negative_correlation}
J$.$ Kahn, and M$.$ Neiman,
\newblock {\it{Negative correlation and log-concavity}},
\newblock \href{http://arxiv.org/abs/0712.3507}{arXiv:0712.3507}

\bibitem{kodama:kp_solitons}
Y$.$ Kodama and L$.$ Williams,
\newblock {\it{KP solitons and total positivity for the Grassmannian}},
\newblock Inventiones Mathematicae, 198 (2014), 637--699.

\bibitem{knutson:positroid}
A$.$ Knutson, T$.$ Lam, and D$.$ Speyer.
\newblock {\it{Positroid varieties: juggling and geometry}},
\newblock Compositio Mathematica, 149 (2013), 1710--1752.

\bibitem{pemantle:negative_dependence}
R$.$ Pemantle.
\newblock {\it{Towards a theorem of negative dependence}},
\newblock Journal of Mathematical Physics, 41 (200), 1371--1390.

\bibitem{postnikov:positive_grassmannian}
A$.$ Postnikov.
\newblock {\it{Total positivity, Grassmannians, and networks}},
\newblock \href{https://arxiv.org/abs/math/0609764}{arXiv:math/0609764}

\bibitem{srinivasan:approximation_algorithms}
A$.$ Srinivasan.
\newblock {\it{Distributions on level sets with applications to approximation algorithms}},
in ``Foundations of Computer Science", (2001), 588--597.

\bibitem{wagner:masons_conjecture}
D$.$ Wagner
\newblock {\it{Mason's conjecture for independent sets in matroids}},
\newblock Annals of Combinatorics, 12 (2008), 211--239.

\bibitem{wagner:electrical_network}
D$.$ Wagner,
\newblock{\it{Matroid inequalities from electrical network theory}},
\newblock Electronic Journal of Combinatorics, 11 (2004-6), A1.

\bibitem{xu:thesis}
Y$.$ Xu.
\newblock {\it{Rayleigh property of lattice path matroids}},
\newblock Master's Thesis, University of Waterloo (2015), \href{https://uwspace.uwaterloo.ca/handle/10012/9694}{https://uwspace.uwaterloo.ca/handle/10012/9694}.

\end{thebibliography}
\end{document}